\documentclass[12pt]{amsart}

\usepackage[a4paper,margin=1in]{geometry}
\usepackage{amsmath,amssymb,amsthm,mathtools}
\usepackage{mathrsfs}
\usepackage{bm,xcolor}

\usepackage[shortlabels]{enumitem}
\usepackage{graphicx}
\definecolor{darkblue}{rgb}{0.0, 0.0, 0.55}
\definecolor{brickred}{rgb}{0.8, 0.25, 0.33}
\definecolor{olivegreen}{rgb}{0.33, 0.42, 0.18}
\usepackage[colorlinks,breaklinks,linkcolor=brickred,citecolor=olivegreen,urlcolor=darkblue,hypertexnames=true]{hyperref}
\usepackage{subcaption}

\usepackage[normalem]{ulem}
\usepackage{cleveref}
\usepackage[baseline]{euflag}

\newtheorem{theorem}{Theorem}
\newtheorem{lemma}[theorem]{Lemma}

\newtheorem{example}[theorem]{Example}

\newcommand{\R}{\mathbb R}
\newcommand{\N}{\mathbb N}
\newcommand{\Z}{\mathbb Z}

\def\M{\mathbf{M}}

\DeclareMathOperator{\supp}{supp}

\DeclareMathOperator{\He}{He}

\title[Convergence Rates for Lasserre's Upper-Bound Hierarchy on $\R^d$]{Convergence Rates for Lasserre's Upper-Bound Hierarchy on $\R^d$ with Gaussian Measure}

\author{Sander Gribling}
\address{Department of Econometrics \& Operations Research, Tilburg University, Netherlands}
\email{s.j.gribling@tilburguniversity.edu}

\author{Igor Klep}
\address{University of Ljubljana, Faculty of Mathematics and Physics, Jadranska 21, 1000 Ljubljana \&
University of Primorska, Faculty of Mathematics, Natural Sciences and Information Technologies, Glagoljaška ulica 8, 6000 Koper, Slovenia.}
\email{igor.klep@fmf.uni-lj.si}
\author{Poornendu Kumar}
\address{University of Ljubljana, Faculty of Mathematics and Physics, Jadranska 21, 1000 Ljubljana}
\email{poornendu.kumar@fmf.uni-lj.si}



\begin{document}

\begin{abstract}
We establish quantitative convergence rates for Lasserre's measure-based upper-bound hierarchy on $\R^d$ using the standard Gaussian measure as the reference measure. For any polynomial $f$ attaining its global infimum $f_{\min}$, we prove that the order-$r$ upper bounds satisfy 
\[
\lambda_r-f_{\min}=O\Big(\frac{1}{r}\Big),
\]
yielding, to the best of our knowledge, the first explicit convergence rate for this hierarchy on a noncompact domain.  We further show that the rate is sharp by proving an $\Omega(1/r)$ rate for $f=x^2$, while demonstrating that, without attainment of the infimum, convergence can be arbitrarily slow: for every $m\ge0$, we construct a polynomial whose upper bounds converge at rate $\Theta(r^{-1/(m+2)})$.
\end{abstract}

\maketitle


\section{Introduction}

Let $\R[x]$ be the space of polynomials and 
let $\R[x]_{r}$ be the space of polynomials with degree at most $r$. 
Given a polynomial $f\in\R[x]$ that is bounded from below, we consider the global polynomial optimization problem
\begin{equation}\label{eq:holygrail}
f_{\min}=\inf_{x\in\R^d} f(x).
\end{equation}
Semidefinite programming (SDP) serves as a cornerstone of polynomial optimization, linking sum-of-squares (SOS) certificates
and Positivstellensätze
from real algebraic geometry \cite{Putinar,BCR,Marshall,Scheiderer} to moment problems \cite{LasserreBook,schmbook,NieBook}. 
This perspective underlies the moment--SOS hierarchies of Lasserre \cite{Las01} (see also Parrilo \cite{Par03} and Laurent's survey \cite{LaurentSOS}), which replace the intractable nonconvex problem \eqref{eq:holygrail} with a hierarchy of solvable semidefinite relaxations.
For the moment--SOS hierarchy of lower bounds for $f_{\min}$, quantitative results have been obtained using polynomial-kernel methods \cite{FangFawzi,gribling26Gaussian,Magron26}, effective Positivstellensatz bounds \cite{Schweighofer,BaldiMourrain,BaldiSlot,BaldiMourrainParusinski}, and correlative sparsity \cite{KordaMagronRios}. 
Related phenomena concerning slow convergence have recently been investigated in \cite{HenrionLeFrancMagron}. 
We refer to \cite{LaurentSlotSurvey} for a recent overview.

In contrast to lower bounds certified by polynomial nonnegativity, Lasserre \cite{lasserre2011new} introduced a complementary hierarchy of \emph{upper} bounds based on measure-theoretic approximations. The idea is to minimize the expectation of $f$ over probability measures possessing a sum-of-squares (SOS) density with respect to a prescribed reference measure. 
Specifically, let $\mu$ be a Borel measure
on $\R^d$ with $\supp\mu=\R^d$, whose moment
sequence satisfies the Carleman condition \cite{schmbook}. 
For each $r \ge \lceil \deg(f)/2 \rceil$, 
the $r$-th upper bound is defined by the SDP
\begin{align}
\label{eq:lassUpper}
\begin{array}{rl}
\lambda_r=\displaystyle\inf_{h \in \R[x]_r}& \,\int_{\R^d} f(x) h^2(x) \, d \mu(x)  \\
\mbox{s.t.}& \int_{\R^d} h^2(x)  \, d \mu(x) = 1.
\end{array}
\end{align}
Equivalently, one may optimize over SOS densities
$\sigma\in\Sigma[x]_{2r}$, where
\[
\Sigma[x]_{2r}
=
\Big\{
\sum_j h_j(x)^2 : h_j\in\R[x]_r
\Big\}.
\]
The resulting sequence $(\lambda_r)_r$ is non-increasing and converges  to
$f_{\min}$ from above \cite{lasserre2011new}.
We refer the reader to \cite{KMV26} for a noncommutative extension of Lasserre's upper-bound hierarchy.

Existing quantitative analyses for Lasserre's upper-bound hierarchy focus almost exclusively on \emph{compact} sets \cite{doherty2012convergence,dKLS17,dKHL17,dKL22,SL22}; see \cite{LaurentSlotSurvey} for an overview. For example, the convergence rate on a convex body with respect to the Lebesgue measure is $O((\log r)^2/r^2)$; the log-factors can be removed for simple sets such as the sphere and hypercube (for the right reference measure). Here, we consider the noncompact setting on $\R^d$ endowed with the standard Gaussian measure. 

The strong tail bounds of the Gaussian measure may suggest a similar behaviour to the compact setting, however, we show there are significant differences.

\subsection{Main result}

 Throughout, the ambient dimension is denoted by $d$. Our main result is an explicit convergence rate for Lasserre’s upper-bound hierarchy when $\mu$ is the standard Gaussian (probability) measure on $\R^d$,
\begin{equation*}
d\mu(x)=(2\pi)^{-d/2}e^{-\|x\|^2/2}\,dx.
\end{equation*}
For $r\in\N$, we consider
Lasserre's upper bound hierarchy \eqref{eq:lassUpper} applied to this measure.

\begin{theorem}
\label{thm:improved-log-rate}
Let $f\in\R [x]$ have degree at most $N$, and assume that $f$
attains its global infimum on $\R ^d$. Then there are constants
$C_{f,d,N}>0$ and $r_0\in\N$ such that, for all $r\ge r_0$,
\begin{equation*}
0\le \lambda_r-f_{\min}
\le
\frac{C_{f,d,N}}{r}.
\end{equation*}
\end{theorem}

We complement this upper bound with examples in Section~\ref{sec:examples} that illustrate both the rate and the role of the
attainment assumption. For the univariate quadratic $f=x^2$, we show that the rate is sharp: the hierarchy converges at the rate $\Theta(1/r)$. A second class of examples, for which the
infimum is not attained, shows the convergence can be arbitrarily slow among
polynomial rates: for every $m\ge0$ we exhibit a polynomial $f$ of degree $2m+4$ for which
\[
\lambda_r - f_{\min}=\Theta\left(r^{-1/(m+2)}\right).
\]
\Cref{fig:errors} further illustrates the performance on two examples: a non-centered quadratic $f(x) = (x-1)^2$, and the Motzkin polynomial $f(x,y) = x^4 y^2 + x^2 y^4 -3 x^2 y^2 + 1$.

\begin{figure}[htbp]
  \centering
  \begin{subfigure}[t]{0.48\textwidth}
    \centering
    \includegraphics[width=\textwidth]{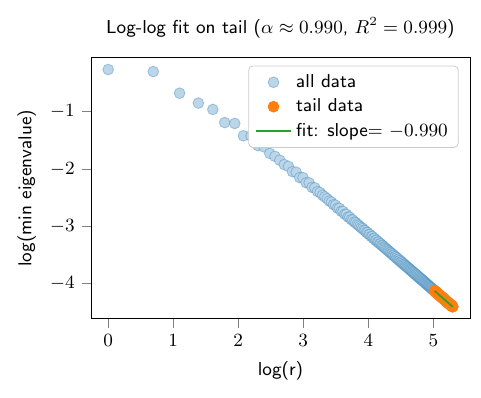}
    \caption{$f(x) = (x-1)^2$}
    \label{fig:errors-a}
  \end{subfigure}
  \hfill
  \begin{subfigure}[t]{0.48\textwidth}
    \centering
   \includegraphics[width=\textwidth]{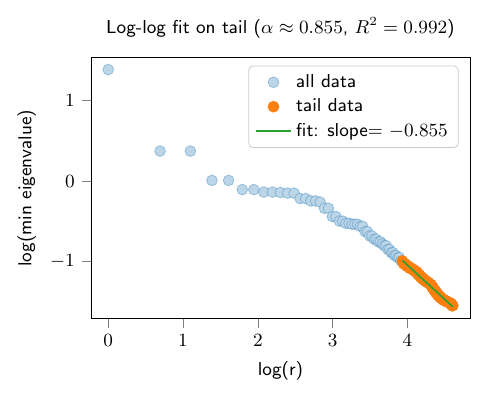}
    \caption{$f(x,y) = x^4 y^2 + x^2 y^4 -3 x^2 y^2 + 1$}
    \label{fig:errors-b}
  \end{subfigure}
  \caption{Log-log plot of the convergence rate on $(x-1)^2$~\eqref{fig:errors-a} ($1 \leq r \leq 200$) and the Motzkin polynomial~\eqref{fig:errors-b} ($1 \leq r \leq 100$). The power-law fit $\propto r^{-\alpha}$ is based on the last $50$ values of $r$.}
  \label{fig:errors}
\end{figure}

\subsection{Guide to the paper}
Section~\ref{sec:preliminaries} states the necessary background on Hermite polynomials. Section~\ref{sec:localization_univariate} contains our key technical lemma, establishing the desired localization in the univariate setting. Section~\ref{sec:proof_main} extends this to the multivariate setting to establish Theorem~\ref{thm:improved-log-rate}. 
Finally, Section~\ref{sec:examples} studies the
hierarchy through two examples: a quadratic example with convergence rate
$\Theta(1/r)$, and a family showing arbitrarily slow polynomial convergence when the infimum is not attained.

\subsection*{Acknowledgments}
This work was performed within the project
COMPUTE, which is funded within the QuantERA II Programme that has received funding from
the EU’s H2020 research and innovation programme under the GA No.~101017733 {\normalsize\euflag}.
IK also acknowledges support of the Slovenian
Research Agency program P1-0222 and grants J1-50002, N1-0217, J1-60011, J1-50001, J1-3004 and
J1-60025.

\section{Preliminaries on Hermite polynomials}\label{sec:preliminaries}

Consider the Gaussian probability measure on $\R$,
\[
    d\mu(x) =  \frac{1}{\sqrt{2\pi}} e^{-\frac{x^2}{2}} \, dx.
\]
With respect to $\mu$ the orthogonal polynomials are the
so--called probabilists' Hermite polynomials $\He_n(x)$ given by
\[
    \He_n(x) = (-1)^n e^{\frac{x^2}{2}} \frac{d^n}{dx^n} \left( e^{-\frac{x^2}{2}} \right).
\]
Then
\[
\int_{-\infty}^{\infty} \He_n(x) \He_m(x)\, d\mu(x) =  n!\, \delta_{nm}
\]
so $h_n(x)=\frac{\He_n(x)}{\sqrt{n!}}$, $n\in\N_0$
is an orthonormal basis of $\R[x]$ w.r.t.~$\mu$.

The Hermite polynomials satisfy the three-term recurrence relation 
\begin{equation}\label{eq:Herecurr}
x \He_m(x) = \He_{m+1}(x) + m \He_{m-1}(x),
\end{equation}
whence
\[
    x^2 \He_m(x) = \He_{m+2}(x) + (2m+1)\He_m(x) + m(m-1)\He_{m-2}(x).
\]
Equivalently, 
\begin{equation*}
\begin{split}
x h_m(x) &= \sqrt{m+1} h_{m+1}(x) + \sqrt{m} h_{m-1}(x), \\
x^2 h_m(x) &= \sqrt{(m+1)(m+2)} h_{m+2}(x) + (2m+1)h_m(x) + \sqrt{m(m-1)}h_{m-2}(x).
\end{split}
\end{equation*}
For these standard facts, see \cite[Section~5.5]{SzegoOrthPoly},
with the change of normalization
$\He_n(x)=2^{-n/2}H_n(x/\sqrt{2})$.

\section{Localization in the univariate setting} \label{sec:localization_univariate}

Here we prove our main technical lemma, which provides the desired localization in the univariate setting. 

\begin{lemma}
\label{lem:higher-order-localization}
Let $\nu$ be the standard Gaussian probability measure on $\R $.
Fix $B>0$ and $L\in\N$. Then there exists a constant
$C=C(B,L)>0$ such that, for every sufficiently large $r$ and every
$a\in\R $ satisfying $
        |a|\leq Br^{1/4},$
there exists  $q\in\R [x]_r$ such that
\[
        \int_{\R }q(x)^2\,d\nu(x)=1 \quad \text {
and} \quad 
        \int_{\R }|x-a|^{2\ell}q(x)^2\,d\nu(x)
        \leq
        \frac{C}{r^\ell},
        \qquad
        1\leq\ell\leq L.
\]
\end{lemma}
\begin{proof}
We construct a nonzero polynomial $Q$ of degree at most $r$ such that
\[
    \|(x-a)^\ell Q\|_{L^2(\nu)}
    \leq Cr^{-\ell/2}\|Q\|_{L^2(\nu)},
    \qquad 1\leq\ell\leq L.
\]
Using the Hermite recurrence, we achieve this by choosing the
coefficients of $Q$ with a slowly varying amplitude and phases
that produce cancellation. Throughout the proof, $C>0$ may change from line to line and depends
only on $B$, $L$, and a fixed cutoff function chosen below.

\medskip
\noindent {\em Step 1:} Construction of the coefficient sequence.\\
Let $\{h_k\}_{k\geq0}$ be the normalized Hermite polynomials, see Section~\ref{sec:preliminaries}. Choose a nonzero real-valued function
$\eta\in C_c^{L+2}\bigl((1/4,3/4)\bigr)$,
and set $\eta_k=\eta(k/r)$. Choose constants
\[
    0<c_0<d_0<d_1<C_0<1
\]
such that $\operatorname{supp}\eta\subset[d_0,d_1]$.
For sufficiently large $r$, all indices within distance $L+2$
of the support of $(\eta_k)$ lie in $[c_0r,C_0r]$.
Thus these indices satisfy $k\asymp r$, with comparison constants
independent of $r$ and $a$. For indices in this interval, define
\[
    \phi_k
    =
    \arccos\left(\frac{a}{\sqrt{k}+\sqrt{k+1}}\right).
\]
This is well defined for sufficiently large $r$, since
$|a|\leq Br^{1/4}$ implies
\[
    \left|\frac{a}{\sqrt{k}+\sqrt{k+1}}\right|
    \leq Cr^{-1/4}\leq \frac12.
\]
Assign arbitrary real values to $\phi_k$ at the remaining
nonnegative indices, and define
\[
    \Phi_0=0,\qquad
    \Phi_{k+1}-\Phi_k=\phi_k.
\]
Set
\[
    c_k=\eta_k e^{i\Phi_k},\qquad 0\leq k\leq r,
    \qquad
    Q(x)=\sum_{k=0}^r c_kh_k(x),
\]
and extend the sequence by setting $c_k=0$ for $k>r$. By orthonormality,
\[
    \|Q\|_{L^2(\nu)}^2=\sum_{k=0}^r\eta_k^2.
\]
Moreover, the Riemann-sum convergence
\[
    \frac1r\sum_{k=0}^r\eta(k/r)^2
    \longrightarrow
    I_\eta:=\int_0^1\eta(t)^2\,dt>0
\]
shows that, for sufficiently large $r$,
\begin{equation}\label{eq:Q-norm}
    c_\eta r
    \leq \|Q\|_{L^2(\nu)}^2
    \leq C_\eta r,
\end{equation}
where $c_\eta=I_\eta/2$ and $C_\eta=2I_\eta$ are positive
constants depending only on $\eta$, independent of $r$ and $a$. 
We next show how the chosen phases produce cancellation
under multiplication by $x-a$.

\medskip
\noindent{\em Step 2:} Cancellation under multiplication by $x-a$.\\
For a finitely supported sequence $u$, define
\[
    (\mathcal J_a u)_j
    =
    \sqrt{j}\,u_{j-1}
    +\sqrt{j+1}\,u_{j+1}
    -au_j,
    \qquad u_{-1}=0.
\]
The Hermite recurrence gives
\[
    (x-a)\sum_{j\geq0}u_jh_j(x)
    =
    \sum_{j\geq0}(\mathcal J_a u)_jh_j(x).
\]
Consequently,
\begin{equation}\label{eq:moment-coefficients}
    \|(x-a)^\ell Q\|_{L^2(\nu)}^2
    =
    \|\mathcal J_a^\ell c\|_{\ell^2}^2.
\end{equation}
We now remove the oscillatory factor from the coefficients.
Writing $u_j=e^{i\Phi_j}v_j$ and
$\Delta v_j=v_{j+1}-v_j$, we obtain
\begin{equation}\label{eq:conjugated-J}
    e^{-i\Phi_j}(\mathcal J_a u)_j
    =
    b_jv_j-\gamma_j\Delta v_{j-1}
    +\beta_j\Delta v_j,
\end{equation}
where
\[
    \gamma_0=0,\qquad
    \gamma_j=\sqrt{j}\,e^{-i\phi_{j-1}}\quad (j\geq1),
\]
and
\[
    \beta_j=\sqrt{j+1}\,e^{i\phi_j},
    \qquad
    b_j=\gamma_j+\beta_j-a,
    \qquad j\geq0.
\]
We use the convention $v_{-1}=0$. The phase was chosen so that the leading terms in $b_j$
cancel. We will use the estimates
\begin{equation}\label{eq:coefficient-differences}
\begin{aligned}
    |\Delta^s\gamma_j|+|\Delta^s\beta_j|
    &\leq Cr^{1/2-s},
        &&0\leq s\leq L+1,\\
    |\Delta^s b_j|
    &\leq Cr^{-1/2-s},
        &&0\leq s\leq L.
\end{aligned}
\end{equation}
These estimates are needed only for indices in the fixed
neighborhood of the support described above. To verify them, consider the smooth functions
\[
    \rho(t)=\frac{a}{\sqrt{t}+\sqrt{t+1}},
    \qquad
    \phi(t)=\arccos \rho(t),
\]
on a slightly larger interval where $t\asymp r$.
Differentiation gives
\[
    |\rho^{(s)}(t)|
    \leq C|a|r^{-1/2-s}
    \leq Cr^{-1/4-s},
    \qquad 0\leq s\leq L+2.
\]
Since $|\rho(t)|\leq1/2$, the chain rule yields
\[
    |\phi^{(s)}(t)|
    +
    \left|\frac{d^s}{dt^s}e^{\pm i\phi(t)}\right|
    \leq Cr^{-1/4-s},
    \qquad 1\leq s\leq L+2.
\]
Thus the smooth extensions 
\[
    \gamma(t)=\sqrt{t}\,e^{-i\phi(t-1)},
    \qquad
    \beta(t)=\sqrt{t+1}\,e^{i\phi(t)}
\]
satisfy
\[
    |\gamma^{(s)}(t)|+|\beta^{(s)}(t)|
    \leq Cr^{1/2-s},
    \qquad 0\leq s\leq L+1.
\]
For the cancellation term, the identity
\[
    (\sqrt{t}+\sqrt{t+1})\cos\phi(t)=a
\]
gives
\[
\begin{aligned}
    b(t)
    &:=
    \gamma(t)+\beta(t)-a=
    \sqrt{t}\bigl(e^{-i\phi(t-1)}-e^{-i\phi(t)}\bigr)
    +i(\sqrt{t+1}-\sqrt{t})\sin\phi(t).
\end{aligned}
\]
The $s$th derivative of the first term is bounded by
$Cr^{-3/4-s}$, using
\[
    e^{-i\phi(t-1)}-e^{-i\phi(t)}
    =
    -\int_0^1
    \frac{d}{dt}e^{-i\phi(t-\tau)}\,d\tau.
\]
The $s$th derivative of the second term is bounded by
$Cr^{-1/2-s}$. Hence
\[
    |b^{(s)}(t)|\leq Cr^{-1/2-s},
    \qquad 0\leq s\leq L.
\]
The finite-difference bounds
\eqref{eq:coefficient-differences} now follow from
\[
    \Delta^sF(j)
    =
    \int_{[0,1]^s}
    F^{(s)}(j+t_1+\cdots+t_s)\,dt_1\cdots dt_s,
    \qquad s\geq1.
\]
We next use these estimates to control the Hermite coefficients
of $(x-a)^\ell Q$ for $1\leq\ell\leq L$.

\medskip

\noindent{\em Step 3}: Iterating the cancellation.\\
Each application of $\mathcal J_a$ enlarges the support by at most
one index on either side, while $\Delta^s$ introduces shifts of
at most $s$ indices. Since $\ell+s\leq L+1$ in the estimates below,
all relevant indices lie in the neighborhood specified in Step~1.
Thus the coefficient bounds from Step~2 apply throughout the induction.

We claim that
\[
    (\mathcal J_a^\ell c)_j
    =
    e^{i\Phi_j}v_j^{(\ell)},
    \qquad 0\leq\ell\leq L,
\]
where
\begin{equation}\label{eq:iterated-difference-bound}
    |\Delta^s v_j^{(\ell)}|
    \leq Cr^{-\ell/2-s},
    \qquad 0\leq s\leq L+1-\ell.
\end{equation}
For $\ell=0$, this follows from
$v_j^{(0)}=\eta(j/r)$ and the smoothness of $\eta$. Suppose the claim holds for some $\ell<L$.
By \eqref{eq:conjugated-J},
\[
    v_j^{(\ell+1)}
    =
    b_jv_j^{(\ell)}
    -\gamma_j\Delta v_{j-1}^{(\ell)}
    +\beta_j\Delta v_j^{(\ell)}.
\]
Apply the discrete Leibniz rule
\[
    \Delta^s(fg)_j
    =
    \sum_{t=0}^s
    \binom{s}{t}
    (\Delta^t f)_{j+s-t}
    (\Delta^{s-t}g)_j.
\]
For $0\leq s\leq L-\ell$, each term arising from
$b_jv_j^{(\ell)}$ is bounded by
\[
    Cr^{-1/2-t}r^{-\ell/2-(s-t)}
    =
    Cr^{-(\ell+1)/2-s}.
\]
Likewise, each term arising from either difference term
is bounded by
\[
    Cr^{1/2-t}r^{-\ell/2-(s-t+1)}
    =
    Cr^{-(\ell+1)/2-s}.
\]
This proves \eqref{eq:iterated-difference-bound} by induction.
In particular, taking $s=0$ gives
\[
    |(\mathcal J_a^\ell c)_j|
    \leq Cr^{-\ell/2}.
\]
By the support observation above, $\mathcal J_a^\ell c$
has at most $Cr$ nonzero entries, and hence
\[
    \|\mathcal J_a^\ell c\|_{\ell^2}^2
    \leq Cr^{1-\ell},
    \qquad 1\leq\ell\leq L.
\]
Combining this with \eqref{eq:Q-norm} and
\eqref{eq:moment-coefficients}, we obtain
\[
    \frac{\displaystyle
        \int_{\mathbb R}|x-a|^{2\ell}|Q(x)|^2\,d\nu(x)}
    {\displaystyle
        \int_{\mathbb R}|Q(x)|^2\,d\nu(x)}
    \leq Cr^{-\ell},
    \qquad 1\leq\ell\leq L.
\]
It remains to obtain a real polynomial satisfying these bounds,
up to a constant factor.

\medskip
\noindent {\em Step 4:} Choosing a real polynomial.\\
Write $Q=Q_1+iQ_2$, with $Q_1,Q_2\in\mathbb R[x]_r$.
Choose $Q_\ast\in\{Q_1,Q_2\}$ such that
\[
    \|Q_\ast\|_{L^2(\nu)}^2
    \geq \frac12\|Q\|_{L^2(\nu)}^2.
\]
Since $Q_\ast(x)^2\leq |Q(x)|^2$, the polynomial $
    q=\frac{Q_\ast}{\|Q_\ast\|_{L^2(\nu)}}$
satisfies
\[
    \int_{\mathbb R}q(x)^2\,d\nu(x)=1,
    \qquad
    \int_{\mathbb R}|x-a|^{2\ell}q(x)^2\,d\nu(x)
    \leq Cr^{-\ell},
    \quad 1\leq\ell\leq L.
\]
This completes the proof.
\end{proof}

\section{Proof of \texorpdfstring{Theorem~\ref{thm:improved-log-rate}}{Theorem 1}} \label{sec:proof_main}

Let $f \in \R[x]$ have degree at most $N$. Assume $f$ attains its global infimum on $\R^d$ at $x^* = (a_1,\ldots,a_d)$. We construct a sequence of degree-$r$ polynomials $q$ such that 
$\int_{\R^d} q(x)^2 d\mu(x)=1$ and 
\[
\int_{\R^d} (f(x)-f_{\min}) q(x)^2 d\mu(x) \leq C_{f,d,N}/r
\]
for some constant $C_{f,d,N}$. 

To do so, we use Lemma~\ref{lem:higher-order-localization} with $B=1$ and $L = \lceil N/2\rceil$. For each $j \in [d]$, it provides a sequence of degree-$\lfloor r/d \rfloor$ univariate polynomials $q_j$ such that, for $r$ large enough, we have 
\[
        \int_{\R }q_j(x_j)^2\,d\nu(x_j)=1 \quad \text {
and} \quad 
        \int_{\R }|x_j-a_j|^{2\ell}q(x_j)^2\,d\nu(x_j)
        \leq
        \frac{C}{r^\ell},
        \qquad
        1\leq\ell\leq L.
\]
Here $C>1$ is a constant and $d\nu$ is the univariate Gaussian measure. In particular, this bounds the even moments of $x_j-a_j$. We first extend this bound to odd moments. Let $1 \leq k = 2\ell-1 \leq 2L$. Then $2\ell/k = 1+1/k>1$ and therefore H\"older's inequality yields 
\begin{align*}
\int_{\R }|x_j-a_j|^{k}q(x_j)^2\,d\nu(x_j)
        &\leq \left(\int_{\R} |x_j-a_j|^{2\ell} q_j(x_j)^2 d\nu(x_j)\right)^{k/2\ell}\left(\int_{\R} q_j(x_j)^2 d\nu(x_j)\right)^{1-k/2\ell} \\
        &\leq \left(C r^{-\ell}\right)^{k/2\ell} \leq C r^{-k/2}.
\end{align*}

We then set 
\[
q(x) = \prod_{j \in [d]} q_j(x_j),
\]
which is a degree-$r$ polynomial since $d \lfloor r/d\rfloor \leq r$. Since the multivariate Gaussian measure equals the product of the univariate measures, we obtain 
\[
        \int_{\R^d }q(x)^2\,d\mu(x)=1 \quad \text {
and} \quad 
        \int_{\R^d } |x-a|^{\alpha} q(x)^2\,d\mu(x)
        \leq
        C r^{-|\alpha|/2},
        \qquad
        1\leq|\alpha|\leq 2L.
\]

To conclude the proof, it remains to observe that since $x^*$ is a global minimizer of $f$, the Taylor expansion of $f-f_{\min}$ around $x^*$ has no constant of linear terms:
\[
f-f_{\min} = \sum_{\alpha \in \N^d:2 \leq |\alpha|\leq N} c_\alpha (x-x^*)^\alpha. 
\]
We therefore have 
\begin{align*}
    \lambda_r - f_{\min} &\leq \sum_{\alpha \in \N^d:2 \leq |\alpha|\leq N} |c_\alpha|  C r^{-|\alpha|/2} \leq \left(C\sum_{\alpha \in \N^d:2 \leq |\alpha|\leq N} |c_\alpha| \right) r^{-1}
\end{align*}
which is of the desired form for $C_{f,d,N} = C\sum_{\alpha \in \N^d:2 \leq |\alpha|\leq N} |c_\alpha|  $.

\section{Examples}\label{sec:examples}

To provide context and examine the sharpness and limitations of Theorem \ref{thm:improved-log-rate}, we first recall the dual SDP formulation of Lasserre’s upper hierarchy \eqref{eq:lassUpper}.
For this we need to introduce moment matrices and localized moment matrices.
We refer to \cite{LaurentSOS,NieBook,HenrionKordaLasserre} for background on moment and localizing matrices and their role in polynomial optimization.

Given a measure $\mu$ on $\R^d$ and $r\in\N$, the truncated commutative  \emph{moment matrix} $\M_r(\mu)$ associated with $\mu$
is the real symmetric matrix with rows and columns indexed by a 
basis (e.g.~the canonical basis $x^\alpha$) 
of $\R[x]_r$ 
and with entries:
\begin{align*}
\M_r(\mu)(\alpha,\beta) =  
\int_{\R^d} x^{\alpha + \beta}\, d\mu, \quad \alpha,\beta\in\N_r^n,
\end{align*}
{where $\N^n_r = \lbrace \alpha\in(\N\cup\{0\})^n\mid \sum_i \alpha_i\leq r\rbrace$. 
This matrix is a  multivariate version of a (univariate) Hankel matrix.} 

Similarly, for all $f \in \R[x]$, the truncated \emph{localizing moment matrix} $\M_r(f \, \mu)$ is the real symmetric matrix with rows and columns indexed by 
a basis (e.g., the canonical one)
of $\R[x]_r$ 
 and with entries 
$$\M_r(f \, \mu)(\alpha,\beta) =  
\int_{\R^d}
f(x) \, x^{\alpha + \beta}\, d\mu, \quad \alpha,\beta\in\N_r^n.$$
Note that 
$\M_r(1 \, \mu)=\M_r(\mu)$. 

With this notation introduced, the dual SDP to \eqref{eq:lassUpper} is
\begin{align}
\label{eq:lassUpper'}
\begin{array}{rl}
\lambda_r = \displaystyle\sup_{\lambda \in \R}& \,\lambda \\
\mbox{s.t.}& \M_{r}(f\, \mu) \, \succeq \, \lambda \, \M_r( \mu) \,,
\end{array}
\end{align}
and is a generalized eigenvalue problem. Further, if one considers the two matrices 
$\M_{r}(f\, \mu) $ and $\M_{r}(\mu)$
indexed by an orthonormal basis of $\R[x]_r$ w.r.t.~$\mu$, then $\M_r(\mu)=I$ and \eqref{eq:lassUpper'}
is in fact equivalent to an eigenvalue problem.

Connections between polynomial optimization, moment matrices, Jacobi matrices, and extremal zeros of orthogonal polynomials have been exploited in \cite{dKL20} in the context of Lasserre's measure-based upper bounds.

\subsection{An example with convergence rate \texorpdfstring{$\frac1r$}{1/r}}

\begin{example}\label{ex:x^2}\rm
Let $d=1$ and consider the Gaussian probability measure on $\R$,
\[
    d\mu(x) =  \frac{1}{\sqrt{2\pi}} e^{-\frac{x^2}{2}} \, dx.
\]
We will consider \eqref{eq:lassUpper'} for $f=x^2$ with respect to the orthonormal basis $\{h_n\}_{n \in \N_0}$ of normalized Hermite polynomials (see Section~\ref{sec:preliminaries}) so that 
 
$\M_r(\mu)=I$.

Using the three-term recurrence relation of the Hermite polynomials~\eqref{eq:Herecurr}, the localizing matrix $\M_r(x^2\,\mu)$ 
with rows and columns indexed by $0,\ldots,r$
is
easily seen to be 
\[
    \M_r(x^2\, \mu)(n, m) = \begin{cases} 
2m + 1 & \text{if } n = m \\ 
\sqrt{(m+1)(m+2)} & \text{if } n = m + 2 \\ 
\sqrt{m(m-1)} & \text{if } n = m - 2 \\
0 & \text{otherwise} 
\end{cases},
\]
for $n,m = 0,1,\ldots,r$;
that is,
\[
    \M_r(x^2\, \mu) = \begin{pmatrix}
1 & 0 & \sqrt{2} & 0 & 0 & \cdots & 0 \\
0 & 3 & 0 & \sqrt{6} & 0 & \cdots & 0 \\
\sqrt{2} & 0 & 5 & 0 & \sqrt{12} & \cdots & 0 \\
0 & \sqrt{6} & 0 & 7 & 0 & \ddots & \vdots \\
0 & 0 & \sqrt{12} & 0 & 9 & \ddots & \sqrt{r(r-1)} \\
\vdots & \vdots & \vdots & \ddots & \ddots & \ddots & 0 \\
0 & 0 & 0 & \cdots & \sqrt{r(r-1)} & 0 & 2r+1
\end{pmatrix}
\]
The solution $\lambda_r$ to \eqref{eq:lassUpper}
(or, equivalently, \eqref{eq:lassUpper'})
is the smallest eigenvalue of $\M_r(x^2\,\mu)$.

In the spectral theory of orthogonal polynomials, the infinite matrix representing the operator of multiplication by $x$ is called the Jacobi matrix $J$ (see, e.g., \cite[Chapter~4]{SimonOP4}). 
For the normalized Hermite polynomials, using \eqref{eq:Herecurr}, $J$ is a tridiagonal matrix,
whose only non-zero entries are 
\[J_{m, m+1} = J_{m+1, m} = \sqrt{m+1} \quad \text { for }  0 \le m . 
\]

Since our matrix represents multiplication by $x^2$, the infinite-dimensional version of $\M_r(x^2\,\mu)$ is exactly the squared Jacobi matrix $J^2$. 
Let $J_r$ denote the principal $(r+1) \times (r+1)$ submatrix of $J$ indexed by $0,\ldots,r$. 
Let $\bm{e}_j\in\R^{r+1}$, $0\leq j\leq r$, denote the corresponding coordinate vectors.
Evaluating $(J_r^2)_{n,m} = \sum_{k=0}^r (J_r)_{n,k} (J_r)_{k,m}$ and comparing this to $\M_r(x^2\mu)$, we obtain
\begin{equation*}
\M_r(x^2\,\mu) = J_r^2 + (r+1) E_{r,r},
\end{equation*}
where $E_{r,r} = \bm{e}_{r} \bm{e}_{r}^T$ is the rank-one projection  onto the boundary coordinate. 

The eigenvalues of $J_r$ correspond precisely to the roots $\gamma$ of the $(r+1)$-th Hermite polynomial (see, e.g., \cite[Theorem 1.32]{GautschiOrthPoly}), which are distributed symmetrically around zero. Squaring these roots yields the eigenvalues of $J_r^2$, whence every non-zero eigenvalue possesses a multiplicity of exactly 2. 
To analyze the asymptotics of the minimal eigenvalue 
\[\lambda_r = \lambda_{\min}(\M_r(x^2\,\mu)),\] it suffices to consider odd $r$ as the sequence $\lambda_r$ is monotone by definition.
Because $\M_r(x^2\, \mu)$ is a rank-one positive-semidefinite perturbation of $J_r^2$, its eigenvalues interlace with those of $J_r^2$ via Cauchy's interlacing theorem \cite[Theorem 4.3.1]{HJ12}.\looseness=-1

Since $r$ is odd, the polynomial $\He_{r+1}(x)$ contains no root at the origin. Thus, the minimum eigenvalue of $J_r^2$ is $\gamma_{\min}^2 > 0$ with multiplicity 2. Because the positive perturbation $(r+1)E_{r,r}$ has rank one, 
it cannot increase the minimal eigenvalue of $J_r^2$. Hence
\[
\lambda_r = \gamma_{\min}^2.
 \]
By classical asymptotic limits for the zeros of Hermite polynomials, the smallest positive zero scales as $\gamma_{\min} = \Theta(1/\sqrt{r})$ as $r \to \infty$ (e.g., \cite[Theorem 6.32]{SzegoOrthPoly}). Squaring this relation delivers the desired asymptotic rate:
\[
\lambda_r = \Theta\left(\frac{1}{r}\right).
\]
This shows that even for a simple univariate quadratic objective function, the Lasserre upper bound hierarchy 
\eqref{eq:lassUpper}
yields a sharp convergence rate of $\frac 1r$.

\end{example}

\subsection{A family with arbitrarily slow polynomial convergence rates}

If one drops the assumption of $f$ having a minimizer (i.e., if $f$ is merely bounded from below without attaining its infimum), then the conclusion of Theorem \ref{thm:improved-log-rate} may fail to hold.

\begin{example}\label{ex:fmxy}\rm
For an integer $m\ge 0$, consider the polynomial
\[
        f_m(x,y) = x^2+(1+y^2)^m(xy-1)^2 .
\]
One can verify that $f_m(x,y) > 0$ for all $x,y \in \R$ and for $\varepsilon>0$ we have $f_m(\varepsilon,1/\varepsilon) = \varepsilon^2$. Hence, the infimum of $f_m$ is equal to $0$, but it is not attained. We show 
\begin{equation}\label{eq:fmxy}
        \lambda_r^{(m)}=\Theta\left(r^{-1/(m+2)}\right),
        \qquad r\to\infty .
\end{equation}

\end{example}
To establish the rate \eqref{eq:fmxy} we require two lemmas.

\begin{lemma}\label{lem:x^2}
For $r\ge0$ and $a\in\R$, define
\begin{equation*}
        \theta_r(a)
        =
        \inf_{\substack{q\in\R[x]_{ r}\\ q\ne0}}
        \frac{\int_\R (x-a)^2q(x)^2\,d\nu(x)}
             {\int_\R q(x)^2\,d\nu(x)}.
\end{equation*}
Then, 
    for all $a\in\R$, we have
\begin{equation*}
\theta_r(a)\ge \displaystyle\frac{1}{4r+2}.
\end{equation*}
\end{lemma}

\begin{proof}
Let $q\in\R[x]_{ r}$ and define
\[
        \psi(x) = (2\pi)^{-1/4}q(x)e^{-x^2/4}.
\]
Then
\[
        \|\psi\|_{L^2(dx)}^2
        =
        \int_\R q(x)^2\,d\nu(x),
\]
and
\[
        \|(x-a)\psi\|_{L^2(dx)}^2
        =
        \int_\R (x-a)^2q(x)^2\,d\nu(x).
\]
Let
\[
        \psi_k(x) = (2\pi)^{-1/4}h_k(x)e^{-x^2/4}.
\]
The functions $\psi_k$ are the normalized Hermite functions. They are eigenfunctions of the harmonic oscillator Hamiltonian $H = -\frac{\mathrm{d}^2}{\mathrm{d}x^2} + \frac{x^2}{4}$, satisfying
\[
        \left(-\frac{d^2}{dx^2}+\frac{x^2}{4}\right)\psi_k
        =
        \left(k+\frac12\right)\psi_k.
\]
Since $q$ has degree at most $r$, $\psi$ belongs to
$\operatorname{span}\{\psi_0,\ldots,\psi_r\}$. Hence
\[
        \|\psi'\|_2^2+\frac14\|x\psi\|_2^2
        =
        \left\langle
        H \psi,\psi
        \right\rangle
        \le
        \left(r+\frac12\right)\|\psi\|_2^2.
\]
In particular,
\begin{equation}\label{eq:tag4}
        \|\psi'\|_2^2
        \le
        \left(r+\frac12\right)\|\psi\|_2^2.
\end{equation}
On the other hand, integration by parts gives
\[
        \|\psi\|_2^2
        =
        -2
        \int_\R (x-a)\psi'(x)\psi(x)\,dx.
\]
Thus, by the Cauchy–Schwarz inequality,
\[
        \|\psi\|_2^2
        \le
        2\|(x-a)\psi\|_2\|\psi'\|_2.
\]
Combining this estimate with \eqref{eq:tag4}, we get
\[
        \|(x-a)\psi\|_2^2
        \ge
        \frac{\|\psi\|_2^4}{4\|\psi'\|_2^2}
        \ge
        \frac{1}{4r+2}\|\psi\|_2^2.
\]
Returning to $q$, this says
\[
        \int_\R (x-a)^2q(x)^2\,d\nu(x)
        \ge
        \frac{1}{4r+2}
        \int_\R q(x)^2\,d\nu(x).
\]
Taking the infimum over $q\in\R[x]_r$ proves the claim.
\end{proof}

\begin{lemma}\label{lem:weighted-Gaussian-localization}
Fix $m\in\N_0$ and $B>0$. Then there exists a constant
$C=C(B,m)>0$ such that, for every sufficiently large $r$ and every
$a\in\R $ satisfying
\[
        |a|\leq Br^{1/4},
\]
there exists a polynomial $q\in\R [x]_r$ with $
        \int_{\R }q(x)^2\,d\nu(x)=1,$ such that
\[
\begin{split}
        \int_{\R }(x-a)^2q(x)^2\,d\nu(x)
        &\leq
        \frac{C}{r},\\
        \int_{\R }(1+x^2)^{m+1}q(x)^2\,d\nu(x)
&        \leq
        C(1+a^2)^{m+1},\\
        \int_{\R }
        (1+x^2)^m(x-a)^2q(x)^2\,d\nu(x)
&        \leq
        \frac{C}{r}(1+a^2)^m.
\end{split}
\]
\end{lemma}

\begin{proof}
Apply Lemma~\ref{lem:higher-order-localization} with $L=m+1$. Then
\[
        \int_{\R }|x-a|^{2j}q(x)^2\,d\nu(x)
        \leq
        \frac{C}{r^j},
        \qquad
        1\leq j\leq m+1.
\]
Since $1+x^2\leq 2(1+a^2)+2(x-a)^2,$ we obtain
\[
        (1+x^2)^{m+1}
        \leq
        C_m
        \left(
        (1+a^2)^{m+1}
        +
        |x-a|^{2m+2}
        \right).
\]
Therefore,
\[
        \int_{\R }(1+x^2)^{m+1}q(x)^2\,d\nu(x)
        \leq
        C(1+a^2)^{m+1}.
\]
Similarly,
\[
        (1+x^2)^m(x-a)^2
        \leq
        C_m
        \left(
        (1+a^2)^m(x-a)^2
        +
        |x-a|^{2m+2}
        \right).
\]
Hence
\[
\begin{aligned}
\int_{\R }
(1+x^2)^m(x-a)^2q(x)^2\,d\nu(x) \leq
\frac{C}{r}(1+a^2)^m
+\frac{C}{r^{m+1}}\leq
\frac{C}{r}(1+a^2)^m.
\end{aligned}
\]
The estimate for $\int (x-a)^2q^2\,d\nu$ is the case $j=1$.
\end{proof}


\begin{proof}[Proof of the rate in Example \ref{ex:fmxy}]
We first prove the lower bound of \eqref{eq:fmxy}. Set
\[
        A(y)=(1+y^2)^m,
        \qquad
        D(y)=1+y^2A(y).
\]
Completing the square in $x$, we obtain
\[
        f_m(x,y)
        =
        D(y)
        \left(
        x-\frac{yA(y)}{D(y)}
        \right)^2
        +
        \frac{A(y)}{D(y)}.
\]
Hence, by the lower estimate of Lemma \ref{lem:x^2}, for every
$h\in\R [x,y]_r$,
\[
\begin{aligned}
\int_{\R }
f_m(x,y)h(x,y)^2\,d\nu(x)
&\geq
\left(
\frac{D(y)}{4r+2}
+
\frac{A(y)}{D(y)}
\right)
\int_{\R }h(x,y)^2\,d\nu(x).
\end{aligned}
\]
Let $R=4r+2$ and $z=1+y^2$. Then $A(y)=z^m$ and
$D(y)=1+(z-1)z^m.$ If $z\leq R^{1/(m+2)}$, then
\[
        \frac{A(y)}{D(y)}
        \geq
        \frac{1}{2z}
        \geq
        \frac12R^{-1/(m+2)}.
\]
If $z\geq R^{1/(m+2)}$, then, for sufficiently large $R$,
\[
        \frac{D(y)}{R}
        \geq
        \frac{(z-1)z^m}{R}
        \geq
        \frac{z^{m+1}}{2R}
        \geq
        \frac12R^{-1/(m+2)}.
\]
Therefore,
\begin{equation}\label{eq:lambdalow}
        \lambda_r^{(m)}
        \geq
        c_mr^{-1/(m+2)}.
\end{equation}

For the upper bound, let
\[
        n=\left\lfloor\frac r2\right\rfloor,
        \qquad
        Y=n^{1/(2m+4)},
        \qquad
        a=Y^{-1}.
\]
Then 
$
a=n^{-1/(2m+4)}\leq 1$
and
$Y=n^{1/(2m+4)}\leq n^{1/4}$.
Lemma~\ref{lem:weighted-Gaussian-localization}, applied with $r=n$
first at the center $a$ and then at the center $Y$, yields polynomials
$q\in\R [x]_n$ and $v\in\R [y]_n$ such that
\begin{equation}\label{eq:l13applied}
\begin{split}
\int_{\R }q(x)^2\,d\nu(x)
=
\int_{\R }v(y)^2\,d\nu(y)
&=
1, \\
\int_{\R }(x-a)^2q(x)^2\,d\nu(x)
&\leq
\frac{C}{n},\\
\int_{\R }(1+y^2)^{m+1}v(y)^2\,d\nu(y)
&\leq
C_m(1+Y^2)^{m+1},\\
\int_{\R }(1+y^2)^m(y-Y)^2v(y)^2\,d\nu(y)
&\leq
\frac{C_m}{n}(1+Y^2)^m.
\end{split}
\end{equation}
Set $h(x,y)=q(x)v(y).$ Then $\deg h\leq 2n\leq r$ and
$\int_{\R ^2} h(x,y)^2 \, d\nu(x)\,d\nu(y) = 1$. Moreover,
from $x^2 \le 2a^2 + 2(x - a)^2,$ we deduce
\[
        \int_{\R ^2}x^2h^2\,d\nu(x)\,d\nu(y)
        \leq
        2a^2+\frac{C}{n}.
\]
Since $aY=1$, we have
$xy-1=y(x-a)+a(y-Y),$
and therefore
\[
 (xy-1)^2
        \leq
        2y^2(x-a)^2+2a^2(y-Y)^2.
\]
Using Fubini's theorem together with the  estimates \eqref{eq:l13applied} above, we obtain
\[
\begin{aligned}
\int_{\R ^2}
(1+y^2)^m(xy-1)^2h^2\,d\nu(x)\,d\nu(y)
\leq
\frac{C_m}{n}(1+Y^2)^{m+1}
+
\frac{C_ma^2}{n}(1+Y^2)^m.
\end{aligned}
\]
Consequently,
\[
\lambda_r^{(m)}
\leq
2a^2+\frac{C}{n}
+
\frac{C_m}{n}(1+Y^2)^{m+1}
+
\frac{C_ma^2}{n}(1+Y^2)^m.
\]
Since $Y=n^{1/(2m+4)}$ and $a=Y^{-1}$, each term on the right-hand side is bounded by
$C_mn^{-1/(m+2)}$. Thus
\begin{equation}\label{eq:lambdahigh}
        \lambda_r^{(m)}
        \leq
        C_mr^{-1/(m+2)}.
\end{equation}
Combining the \eqref{eq:lambdalow} and \eqref{eq:lambdahigh} gives
\[
        \lambda_r^{(m)}
        =
        \Theta\left(r^{-1/(m+2)}\right). \qedhere
\]
\end{proof}

\section{Conclusion and discussion}

We have established an explicit convergence rate for Lasserre's measure-based
upper-bound hierarchy on the noncompact domain $\R^d$ when the
reference measure is Gaussian. More precisely, if $f\in\R[x]$
attains its global infimum, then
\[
    \lambda_r-f_{\min}
    =
    O\left(\frac{1}{r}\right).
\]
To the best of our knowledge, this is the first quantitative convergence
rate for this hierarchy in a noncompact setting.

The example $f=x^2$ shows that the upper bound cannot be improved for even the simplest polynomial objective:

in that case,
\[
    \lambda_r=\Theta(1/r).
\]

The family of polynomials in Section~\ref{sec:examples} shows, on the other hand, that attainment of the
infimum is a necessary assumption. If the infimum is
not attained, the convergence may occur at rates\looseness=-1
\[
    \Theta\left(r^{-1/(m+2)}\right),
    \qquad m\ge0,
\]
and hence may be arbitrarily slow among polynomial rates. Consequently, some
form of control on the location of near-minimizers is indispensable for
uniform quantitative results of the type proved in
Theorem~\ref{thm:improved-log-rate}. It would be interesting to identify
conditions weaker than attainment that still guarantee a prescribed
convergence rate.

\subsection{On the use of LLMs}
We acknowledge the use of ChatGPT Astra during the derivation of our main bounds. While our initial argument, preserved in Appendix \ref{appdx}, was formulated without automated assistance, discussions with Astra prompted us to recognize that our own Lemma \ref{lem:higher-order-localization} 
could be leveraged more broadly to remove the $\log$-factor. The final proof structure presented in the main text was rigorously checked and written by the authors.

\appendix

\section{Gaussian localization approach}\label{appdx}

Here we give an alternative, perhaps more intuitive, argument that yields an upper bound of $O(\frac{\log r}{r})$ for the upper bound hierarchy with respect to the standard Gaussian measure. 
In a nutshell, we introduce a Gaussian localization centered at a global minimizer $x^*$ of $f$. We then approximate the Gaussian localization by a sum-of-squares density. Our proof uses some results from approximation theory and Gaussian tail bounds, which we state first. 

For a polynomial $
p(x)=\sum_{\mu} p_\mu x^\mu,$ we define the $\ell^1$-norm of its coefficients by
\[
\|p\|_{1}
:=
\sum_{\mu}|p_\mu|.
\]

We use the following polynomial approximation result for the exponential function to construct the ``polynomialized'' Gaussian.

\begin{theorem}[\protect{\cite[Theorem 4.1]{SVapprox}}]\label{thm:approxexp}
For every $b>0$ and $0 < \delta \leq 1$, there exists a polynomial $S_{b,\delta}$ that satisfies 
\[
\sup_{t \in [0,b]} |e^{-t} - S_{b,\delta}(t)| \leq \delta,
\]
and has degree $O(\sqrt{\max\{b,\log(1/\delta)\}\cdot \log(1/\delta)})$.
\end{theorem}

We need the following elementary growth estimate, see, e.g.~\cite[Section 2.7, (2.37)]{RivlinCheb}.

\begin{lemma}\label{lem:cheb-tail}
Let $p\in\R [u]$ have degree at most $s$, and suppose that
\[
\|p\|_{L^\infty[0,1]}\le B .
\]
Then, for every $u\ge 1$,
\[
|p(u)|\le B\,T_s(2u-1)\le B(4u)^s,
\]
where $T_s$ is the Chebyshev polynomial of the first kind; see \cite{MasonHandscomb}.
\end{lemma}

We now state the two tail estimates used in this argument. The first controls polynomial moments of a rescaled Gaussian, while the second controls the polynomial approximation outside its localization region.
{This estimate follows by switching to polar coordinates and using the asymptotics of the incomplete gamma function~\cite[Eq.~(6.5.32)]{AS64handbook}.}

\begin{lemma}\label{lemma_GaussianTail}
Let $\mathtt{a}\in\Z_+^d$. There exists a constant
$C_{d,|\mathtt{a}|}>0$ such that, whenever $\beta R^2\ge 1$,
\begin{align}\label{equation_3_0}
\int_{\|x\|>R} |x^\mathtt{a}|
e^{-\frac{\beta}{2}\|x\|^2}\,dx
\le
C_{d,|\mathtt{a}|}\,\beta^{-1}
R^{d+|\mathtt{a}|-2}
e^{-\frac{\beta R^2}{2}}.
\end{align}
Consequently, if $|\mathtt{a}|\le N$ and $R\ge 1$, then
\[
\int_{\|x\|>R} |x^\mathtt{a}|
e^{-\frac{\beta}{2}\|x\|^2}\,dx
\le
C_{d,N}\,\beta^{-1}
R^{d+N-2}
e^{-\frac{\beta R^2}{2}}.
\]
\end{lemma}
Using the Chebyshev growth estimate we obtain exponential tail control for the resulting polynomialized Gaussian, uniformly over bounded shifts.

\begin{lemma}\label{lem:new8shifted}
Let $g\in\R [x]$ have degree at most $N$, and let $B>0$.
Suppose
\[
q(x)=p\!\left(\frac{\|x-a\|^2}{R^2}\right),
\qquad
\deg p=s,
\qquad
\|p\|_{L^\infty[0,1]}\le 2,
\]
where $\|a\|\le B$.
Then there are constants $K\geq1$ and $c,C>0$, depending only on
$d,N,B$, such that whenever
$
R^2\ge K(s+1),
$
\[
\int_{\|x-a\|>R}|g(x)|q(x)^2\,d\mu(x)
\le
C\|g\|_1 e^{-cR^2},
\]
where $\mu$ denotes the standard Gaussian probability measure.
\end{lemma}
\begin{proof}
Put $y=x-a$ and $r=\|y\|$. By squaring the triangle inequality $\|y\|\leq \|y+a\|+\|a\|$,
we have
$\|y+a\|^2
\ge \frac12\|y\|^2-\|a\|^2$,
whence 
\[
e^{-\|y+a\|^2/2}
\le
e^{B^2/2}e^{-\|y\|^2/4}
\]
since $\|a\|\le B$.
We also have an upper bound on $g$ in terms of the coefficient norm:
\[
|g(y+a)|\le A_{d,N,B}\|g\|_1(1+r)^N,
\]
for a suitable constant $A_{d,N,B}$.

For $r>R$, Lemma~\ref{lem:cheb-tail} and the inequality
$\log(4u)\le u\log4$ for $u\ge1$ give
\[
q(y+a)^2
\le 4(4r^2/R^2)^{2s}
\le 4\exp\!\left(\frac{2s\log4}{R^2}r^2\right)
\le 4e^{r^2/8},
\]
provided $R^2\ge16(\log4)(s+1)$. Consequently,
\begin{align*}
\int_{\|x-a\|>R}|g(x)|q(x)^2\,d\mu(x)
&\le C_{d,N,B}\|g\|_1
   \int_{\|y\|>R}(1+\|y\|)^N e^{-\|y\|^2/8}\,dy\\
&\le C_{d,N,B}\|g\|_1 e^{-R^2/16}
   \int_{\mathbb R^d}(1+\|y\|)^N e^{-\|y\|^2/16}\,dy.
\end{align*}
The last integral is finite and depends only on $d,N$.
Absorbing it into the constant proves the claim with
$K=16\log4$ and $c=1/16$.
\end{proof}

\subsection{Approximating localized Gaussians with polynomial densities}

Let $x^* \in \R ^d$ be a global minimizer of $f$, and write $f_{\min} = f(x^*)$. By definition of the $r$-th Lasserre upper bound $\lambda_r$ in \eqref{eq:lassUpper}, for every nonzero test polynomial $q \in \R [x]$ with $\deg(q) \le r$,
\begin{equation}\label{eq:rayleigh-quotient}
0 \le \lambda_r - f_{\min} \le \frac{\int_{\R ^d} \bigl(f(x) - f_{\min}\bigr) q(x)^2 \, d \mu(x)}{\int_{\R ^d} q(x)^2 \, d \mu(x)}.
\end{equation}
The argument proceeds by constructing a sequence of sum-of-squares weights $q(x)^2$ that approximate a narrow Gaussian kernel centered near $x^*$, and bounding the numerator and denominator in \eqref{eq:rayleigh-quotient}. The narrow Gaussian kernel that we use is as follows. 
Fix $\alpha>0$, and define
\begin{equation*}
\gamma=1+\alpha^2,
\qquad
\beta=1+\frac1{\alpha^2}.
\end{equation*}
We then 
consider the Gaussian localization integral
\begin{align}
I_\alpha(f) &= \int_{\R^d} f(x) \exp\left(-\frac{\|x-\gamma x^\ast\|^2}{2\alpha^2}-\frac{\|x\|^2}{2}\right)\,dx. \notag \\
&=
e^{-\frac{\gamma}{2}\|x^\ast\|^2}
\int\limits_{\R^d}
f(x)e^{-\frac{\beta}{2}\|x-x^\ast\|^2}\,dx. \notag
\\
&=
e^{-\frac{\gamma}{2}\|x^\ast\|^2}
\int_{\R^d}
f(x+x^\ast)e^{-\frac{\beta}{2}\|x\|^2}\,dx. \label{equation_our_f}
\end{align}
Here, the first equality uses the following completion-of-squares identity
\begin{equation} \label{eq:completion-of-squares}
\frac{\|x-\gamma x^\ast\|^2}{2\alpha^2}
+
\frac{\|x\|^2}{2}
=
\frac{\beta}{2}\|x-x^\ast\|^2
+
\frac{\gamma}{2}\|x^\ast\|^2,
\end{equation}
and the second equality uses the change of variables $ x\mapsto x+x^\ast$. One can show that, as $\alpha \to 0$, we have $I_\alpha(f)/I_\alpha(1) \to f(x^*)$.

\medskip
\noindent\textit{Step 1:} Construction and feasibility of the test polynomial.\\
Fix constants $A > \frac{d}{2} + 1$ and $K \ge 1$ (to be chosen sufficiently large). For $r \in \N$, set
\[
\delta = r^{-A}, \quad L = \log(1/\delta) = A \log r, \quad R = \sqrt{Kr},
\]
and define the parameters
\[
\alpha = C_{\deg} \frac{R\sqrt{L}}{r} = C_{\deg} \sqrt{AK \frac{\log r}{r}}, \quad \gamma = 1 + \alpha^2, \quad \beta = 1 + \frac{1}{\alpha^2},
\]
where $C_{\deg} > 0$ is a constant to be fixed shortly. Note that $\alpha \to 0$ and $\beta =1+ \alpha^{-2} = \Theta(r / \log r)$ as $r \to \infty$.
Consider the target Gaussian localization profile 
\[
G_\alpha(x) = \exp\left( - \frac{\beta - 1}{4} \|x - \gamma x^*\|^2 \right).
\]
On the Euclidean ball $B(\gamma x^*, R)$, the exponent satisfies $\frac{\beta - 1}{4} \|x - \gamma x^*\|^2 \in [0, b]$, where
\[
b = \frac{\beta - 1}{4} R^2 = \frac{R^2}{4\alpha^2} = \frac{r^2}{4 C_{\deg}^2 L}.
\]
For all sufficiently large $r$, we have $b \ge L$. Applying Theorem~\ref{thm:approxexp} with interval $[0, b]$ and precision $\delta$, there exists a univariate polynomial $S_{b,\delta}$ such that
\[
\sup_{t \in [0,b]} |e^{-t} - S_{b,\delta}(t)| \le \delta,
\]
with degree bounded by
$s = \deg(S_{b,\delta}) \le C \sqrt{b L} = \frac{C}{2 C_{\deg}} r$.
We define our test polynomial by
\[
q(x) = S_{b,\delta}\left( \frac{\beta - 1}{4} \|x - \gamma x^*\|^2 \right).
\]
Choosing $C_{\deg} \ge C$ ensures that $\deg(q) = 2s \le r$, so $q \in \R [x]_r$ is feasible for the $r$-th relaxation. Moreover, $|q(\gamma x^*) - 1| = |S_{b,\delta}(0) - 1| \le \delta < 1$ for large $r$, so $q \neq 0$.

\medskip
\noindent\textit{Step 2:} Comparison between $q^2 d \mu$ and the localized Gaussian $G_\alpha^2 d \mu$.\\
On the ball $B(\gamma x^*, R)$, we have $|G_\alpha(x) - q(x)| \le \delta$. Since $0 < G_\alpha(x) \le 1$, it follows that $|q(x)| \le 1 + \delta \le 2$, and hence
\[
|q(x)^2 - G_\alpha(x)^2| \le |q(x) - G_\alpha(x)| \cdot |q(x) + G_\alpha(x)| \le 3\delta.
\]
Therefore, for any polynomial $g \in \R [x]$ of degree at most $N$,
\begin{equation} \label{eq:3split}
\begin{split}
\left| \int_{\R ^d} g q^2 \, d \mu - \int_{\R ^d} g G_\alpha^2 \, d \mu \right|
&\le 3\delta \int_{\|x - \gamma x^*\| \le R} |g(x)| \, d \mu(x) \\
&\quad + \int_{\|x - \gamma x^*\| > R} |g(x)| q(x)^2 \, d \mu(x) \\
&\quad + \int_{\|x - \gamma x^*\| > R} |g(x)| G_\alpha(x)^2 \, d \mu(x).
\end{split}
\end{equation}
We bound the three terms on the right-hand side of \eqref{eq:3split} as follows:
\begin{enumerate}[\rm(i)]
\item \emph{Ball approximation term:} Since polynomials are integrable against the Gaussian measure,
\[
3\delta \int_{\|x - \gamma x^*\| \le R} |g(x)| \, d \mu(x) \le 3\delta \|g\|_{L^1(\mu)} \le C_{d,N} \|g\|_1 \delta.
\]
\item \emph{Polynomialized tail:} Set $p(t) = S_{b,\delta}(bt)$ for $t \in [0,1]$. Then $\|p\|_{L^\infty([0,1])} \le 1 + \delta \le 2$ and $q(x) = p\bigl( \|x - \gamma x^*\|^2 / R^2 \bigr)$. Since $\gamma = 1 + \alpha^2 \to 1$, the shift vector satisfies $\|\gamma x^*\| \le 2 \|x^*\|$. By choosing $K \ge 1$ sufficiently large so that $R^2 = Kr \ge K'(s+1)$, Lemma~\ref{lem:new8shifted} implies
\[
\int_{\|x - \gamma x^*\| > R} |g(x)| q(x)^2 \, d \mu(x) \le C_{d,N,x^*} \|g\|_1 e^{-c R^2}.
\]
\item \emph{Gaussian tail:} By Equation~\eqref{eq:completion-of-squares}, the Gaussian weight $G_\alpha$ satisfies
\[
G_\alpha(x)^2 \, d \mu(x) = (2\pi)^{-d/2} e^{-\frac{\gamma}{2}\|x^*\|^2} e^{-\frac{\beta}{2}\|x - x^*\|^2} \, \mathrm{d}x.
\]
Under the substitution $u = x - x^*$, the condition $\|x - \gamma x^*\| > R$ translates to $\|u - \alpha^2 x^*\| > R$. Since $\alpha^2 \|x^*\| \to 0$, for all sufficiently large $r$ this implies $\|u\| > R/2$. Applying Lemma~\ref{lemma_GaussianTail} 
to the polynomial $u\mapsto g(u+x^*)$,
with radius $R/2$ and scaling parameter $\beta$ yields
\[
\int_{\|x - \gamma x^*\| > R} |g(x)| G_\alpha(x)^2 \, d \mu(x) \le C_{d,N,x^*} \|g\|_1 \beta^{-1} R^{d+N-2} e^{-\frac{1}{8}\beta R^2}.
\]
\end{enumerate}
Combining (i)--(iii), there exists a constant $C_g > 0$, depending on $d, N, x^*$, and $\|g\|_1$, such that
\begin{equation}\label{eq:i-iii}
\left| \int_{\R ^d} g q^2 \, d \mu - \int_{\R ^d} g G_\alpha^2 \, d \mu \right| \le C_g \left( \delta + e^{-c R^2} + e^{-c' \beta R^2} \right) = \mathcal{E}_g(r).
\end{equation}
Here, we have absorbed polynomial factors into the exponential. Indeed, since $\beta R^2\to\infty$, after decreasing the exponential constant if necessary, for all sufficiently large $r$,
\[
\beta^{-1}R^{d+N-2}e^{-\beta R^2/8}
\le C e^{-c'\beta R^2}.
\]

\medskip
\noindent\textit{Step 3}: Normalization and bounding the denominator.\\
Applying the estimate \eqref{eq:i-iii} with $g = 1$ gives
\[
\left| \int_{\R ^d} q(x)^2 \, d \mu(x) - J_0 \right| \le \mathcal{E}_1(r),
\]
where
$J_0 = \int_{\R ^d} G_\alpha(x)^2 \, d \mu(x) = (2\pi)^{-d/2} e^{-\frac{\gamma}{2}\|x^*\|^2} \int_{\R ^d} e^{-\frac{\beta}{2}\|u\|^2} \, \mathrm{d}u = e^{-\frac{\gamma}{2}\|x^*\|^2} \beta^{-d/2}$.
We examine the relative error $\mathcal{E}_1(r) / J_0$. Since $\gamma \to 1$, the factor $e^{\frac{\gamma}{2}\|x^*\|^2}$ is bounded by a constant for large $r$. Hence, there exists a constant $C_0 > 0$ such that
\[
\frac{\mathcal{E}_1(r)}{J_0} \le C_0 \beta^{d/2} \left( \delta + e^{-c R^2} + e^{-c' \beta R^2} \right).
\]
Since $\beta = \Theta(r / \log r)$, we estimate the three terms:
\begin{align*}
\beta^{d/2} \delta &\le C \left(\frac{r}{\log r}\right)^{d/2} r^{-A} = o(1) \quad \text{since } A > \frac{d}{2}, \\
\beta^{d/2} e^{-c R^2} &\le C \left(\frac{r}{\log r}\right)^{d/2} e^{-c K r} = o(1), \\
\beta^{d/2} e^{-c' \beta R^2} &\le C \left(\frac{r}{\log r}\right)^{d/2} \exp\left( -c'' \frac{r^2}{\log r} \right) = o(1).
\end{align*}
Thus $\mathcal{E}_1(r) / J_0 = o(1)$. In particular, for all sufficiently large $r$, $\mathcal{E}_1(r) \le \frac{1}{2} J_0$, which establishes the denominator lower bound
\begin{equation}\label{eq:denom-lower-bound}
\int_{\R ^d} q(x)^2 \, d \mu(x) \ge \frac{1}{2} J_0 = \frac{1}{2} e^{-\frac{\gamma}{2}\|x^*\|^2} \beta^{-d/2}.
\end{equation}

\medskip
\noindent\textit{Step 4:} Bounding the numerator.\\
Now apply \eqref{eq:i-iii} to the nonnegative polynomial $g(x) = f(x) - f_{\min}$. Setting
\[
J_F := \int_{\R ^d} \bigl(f(x) - f_{\min}\bigr) G_\alpha(x)^2 \, d \mu(x),
\]
we have $\int_{\R ^d} (f - f_{\min}) q^2 \, d \mu \le J_F + \mathcal{E}_{f - f_{\min}}(r)$. 
Using Equation~\eqref{eq:completion-of-squares} and substituting $u = x - x^*$,
\[
J_F = (2\pi)^{-d/2} e^{-\frac{\gamma}{2}\|x^*\|^2} \int_{\R ^d} F(u) e^{-\frac{\beta}{2}\|u\|^2} \, \mathrm{d}u,
\]
where $F(u) := f(x^* + u) - f(x^*)$. Because $x^*$ is a global minimizer of $f$, we have
\[
F(0) = 0 \quad \text{and} \quad \nabla F(0) = \nabla f(x^*) = 0.
\]
Consequently, the Taylor expansion of $F$ at the origin contains neither a constant nor linear terms, i.e.,
\[
F(u) = \frac{1}{2} u^T \nabla^2 f(x^*) u + \sum_{3 \le |\nu| \le N} \frac{D^\nu f(x^*)}{\nu!} u^\nu.
\]
By symmetry of the Gaussian measure $e^{-\frac{\beta}{2}\|u\|^2} \, \mathrm{d}u$, all odd-degree monomials integrate to zero. For every even monomial of degree $|\nu| = 2k \ge 2$,
\[
\frac{\int_{\R ^d} u^\nu e^{-\frac{\beta}{2}\|u\|^2} \, \mathrm{d}u}{\int_{\R ^d} e^{-\frac{\beta}{2}\|u\|^2} \, \mathrm{d}u} = O\left(\beta^{-k}\right) \le O\left(\beta^{-1}\right).
\]
Summing over the finitely many terms in $F$, we obtain
\begin{equation}\label{eq:gauss-ratio}
\frac{J_F}{J_0} = \frac{\int_{\R ^d} F(u) e^{-\frac{\beta}{2}\|u\|^2} \, \mathrm{d}u}{\int_{\R ^d} e^{-\frac{\beta}{2}\|u\|^2} \, \mathrm{d}u} \le \frac{A_{f,d}}{\beta},
\end{equation}
for a positive constant $A_{f,d}$ depending only on $f$ and $d$.

\medskip
\noindent\textit{Step 5:} Combining the estimates.\\
Combining the Rayleigh quotient \eqref{eq:rayleigh-quotient} with the denominator bound \eqref{eq:denom-lower-bound} and numerator estimate gives
\begin{align}\label{eq:final-rate}
\lambda_r - f_{\min} 
\le \frac{J_F + \mathcal{E}_{f - f_{\min}}(r)}{\frac{1}{2} J_0} &= 2 \frac{J_F}{J_0} + 2 \frac{\mathcal{E}_{f - f_{\min}}(r)}{J_0}.
\end{align}
By \eqref{eq:gauss-ratio} and the definition of $\beta = 1 + \alpha^{-2}$, the leading term satisfies
\[
2 \frac{J_F}{J_0} \le \frac{2 A_{f,d}}{\beta} \le 2 A_{f,d} \alpha^2 = 2 A_{f,d} C_{\deg}^2 A K \frac{\log r}{r}.
\]
For the remaining error term, since $A > \frac{d}{2} + 1$, the same asymptotics as in Step~3 show that
\[
\frac{\mathcal{E}_{f - f_{\min}}(r)}{J_0} \le C \beta^{d/2} \left( r^{-A} + e^{-c K r} + e^{-c'' \frac{r^2}{\log r}} \right) = O\left( \left(\frac{r}{\log r}\right)^{d/2} r^{-A} \right) = o\left( \frac{\log r}{r} \right).
\]
Substituting these bounds into \eqref{eq:final-rate} yields
\[
\lambda_r - f_{\min} \le C_{f,d,N} \frac{\log r}{r}
\]
for all sufficiently large $r$.

\end{document}